\documentclass[10pt]{article}
\usepackage[english]{babel}
\usepackage[dvips]{color}
\usepackage{graphicx, amsmath, amssymb, amsfonts, amsthm}
\usepackage{slashbox}
\usepackage{color}

\newtheorem{theorem}{Theorem}[section]

\newtheorem{lemma}[theorem]{Lemma}

\newtheorem{remark}[theorem]{Remark}

\newcommand{\refpart}[1]{{\it (#1)}}                 
\newcommand{\ppt}{\hspace{1pt}}                   

\newcommand{\heun}[5]{\mbox{\rm Hn}\!\left( {#1 \atop #2} \left|  {#3   \atop #4} \right|\, #5 \right) }

\newcommand{\hpgo}[2]{{}_{#1}\mbox{\rm F}_{\!#2}}
\newcommand{\hpg}[5]{{}_{#1}\mbox{\rm F}_{\!#2}\!
  \left(\left.{#3 \atop #4}\right|\, #5 \right) }

\newcommand{\hpgde}[1]{E(#1)}                         
\newcommand{\heunde}[1]{\mbox{\it HE\ppt}(#1)}  

\newcommand{\degr}{D}    
\newcommand{\la}{\alpha}                                     
\newcommand{\lb}{\beta}

\newcommand{\pback}[1]{\stackrel{\ppt#1}{\longleftarrow}}   

\newcommand{\PP}{{\Bbb P}}
\newcommand{\ZZ}{{\Bbb Z}}
\newcommand{\CC}{{\Bbb C}}

\newcommand{\QQ}{{\Bbb Q}}

\newcommand{\equal}{&\!\!=\!\!&}

\begin{document}


\title{Degenerate and dihedral Heun functions with parameters}

\author{Raimundas Vidunas\footnote{Lab of Geometric \& Algebraic Algorithms,
        Department of Informatics \& Telecommunications,
        National Kapodistrian University of Athens,
        Panepistimiopolis 15784, Greece. E-mail: {\sf rvidunas@gmail.com}.
        Supported by JSPS grant 20740075, and 
        the EU (European Social Fund) and Greek National Fund through 
        the operational program ``Education and Lifelong Learning" of the
        National Strategic Reference Framework, research funding program
        ``ARISTEIA'', project ``ESPRESSO: Exploiting Structure in 
        Polynomial Equation and System Solving in Physical Modeling".}}
\date{\today}
\maketitle

\begin{abstract}
Just as with the Gauss hypergeometric function, particular cases 
of the local Heun function can be Liouvillian (that is, ''elementary") functions. 
One way to obtain these functions is by pull-back transformations
of Gauss hypergeometric equations with Liouvillian solutions.
This paper presents the Liouvillian solutions of Heun's equations
that are pull-backs of the parametric hypergeometric equations with 
cyclic or dihedral monodromy groups.
\end{abstract}

\section{Introduction}

This paper presents all pull-back transformations to Heun equations from Gauss
hypergeometric equations  with {\em Liouvillian solutions} and a continuous parameter. 
This gives interesting series of examples of Liouvillian Heun functions.
A Liouvillian function lies in a {\em Liouvillian extension} of $\CC(x)$, that is \cite{Kovacic},
an extension of differential fields generated by sequentially adjoining 
a finite number of integrals, exponentials of integrals and algebraic functions.

The Heun equation
\begin{equation}
\label{eq:HE}
\frac{{\rm d}^2Y(x)}{{\rm
d}x^2}+\biggl(\frac{c}{x}+\frac{d}{x-1}+\frac{a+b-c-d+1}{x-t}\biggr)\frac{{\rm
d}Y(x)}{{\rm d}x}+\frac{ab\,x-q}{x(x-1)(x-t)} Y(x)=0
\end{equation}
is a canonical second-order Fuchsian differential equation on the Riemann sphere~$\PP^1$
with $4$ regular singularities. The singular points and the local exponents 
are usefully encoded by the Riemann P-symbol
\begin{equation}
\label{eq:PsymbolHE}
P\left\{
\begin{array}{cccc|c}
0 & 1 &t & \infty &x\\ \hline
0 & 0 &0 & a & \\
1-c & 1-d & c+d-a-b & b &
\end{array}
\right\}.
\end{equation}
The local solution at $x=0$ with the local exponent $0$ and the value $1$
of Heun's equation is denoted by
\begin{equation} \label{eq:HF}
\heun{\,t\,}{q}{\,a,\,b\,}{c;\,d}{\,x\,}.
\end{equation}
Generally, the Heun functions are transcendental, and their monodromy is not known.
But they can be Liouvillian for special values of the parameters $a,b,c,d,q,t$. 
The Kovacic algorithm \cite{Kovacic} implies that Heun's equation has 
Liouvillian solutions if and only if its monodromy is either reducible, 
or infinite dihedral, or finite.
The most straightforward case is a {\em Heun polynomial}; then necessarily
$a$ or $b$ equals zero or a negative integer, and the monodromy representation of 
the Heun equation is reducible. 

The simpler Gauss hypergeometric equation
\begin{equation}
\label{eq:GHE}
\frac{{\rm d}^2y(z)}{{\rm d}z^2}+
\left(\frac{C}{z}+\frac{A+B-C+1}{z-1}\right)\,\frac{{\rm d}y(z)}{{\rm d}z}+\frac{AB}{z\,(z-1)}\,y(z)=0
\end{equation}
and its hypergeometric $\hpgo21$ solutions are much better understood. 
Some Heun equations are pull-back transformations
\begin{equation}
\label{eq:algtransf}
z\longmapsto\varphi(x), \qquad y(z)\longmapsto
Y(x)=\theta(x)\,y(\varphi(x)),
\end{equation}
of a specialized Gauss hypergeometric equation. 
Here $\varphi(x)$ is a rational function and $\theta(x)$ is a radical
function ( i.e., a product of powers of rational functions.) 
Monodromy of those Heun equations is then easily computed from
the monodromy of the hypergeometic equations to which a pull-back is applied.
Systematic classification of these pull-back transformations to Heun equations
from hypergeometric equations with a continuous parameter but {\em no Liouvillian solutions}
has been done in \cite{Vidunas2009b}. In total, there are 61 such transformations
up to fractional-linear transformations (of both hypergeometric and Heun equations) and algebraic conjugation. For example, the quadratic transformation P1 puts no restrictions on the parameters
of the $\hpgo21$ function, and gives the formula
\begin{eqnarray} \label{H-2F1 1tr}
\heun{-1}{0}{2a,\,2b}{2c-1;\,a+b-c+1}{\,x} \equal
\hpg{2}{1}{a,\,b\,}{c}{\,x^2}.
\end{eqnarray}
The 61 transformations employ 48 different rational functions $\varphi(x)$ up to M\"obius transformations
and algebraic conjugation \cite{HeunClass}. 
Those rational functions are all {\em Belyi functions} (or {\em Belyi coverings} $\PP^1\to\PP^1$)
as they branch (that is, $\varphi'(x)=0$) only above $\varphi(x)\in\{0,1,\infty\}$.
The branching fibers are exactly the singularities  $z=0$, $z=1$, $z=\infty$ of the hypergeometric equation.

This paper complements \cite{HeunClass}, \cite{Vidunas2009b} by showing all
pull-back transformations to Heun equations from hypergeometric equations 
with Liouvillian solutions and a continuous parameter. 
The hypergeometric-to-Heun pull-back transformations 
with no parameters and no Liouvillian solutions are classified in \cite{VidunasHoeij}. 
There are 366 Galois orbits up to M\"obius transformations. 
Liouvillian Heun functions with no parameters include algebraic Heun functions
that are not classified yet. Kleinian pull-back transformations \cite{Klein1877}
for their Heun equations (with finite monodromy) are of particular interest.

\section{Recalling the classification of pull-backs}

Before recalling the classification scheme  \cite{HeunClass} of hypergeometric-to-Heun transformations,
let us introduce some notation used there (and in \cite{Vidunas2009b}).  
Let $E(\alpha,\beta,\gamma)$ denote a hypergeometric equation with the exponent differences
$\alpha,\beta,\gamma$ assigned to the singular points in any order.
The exponent differences for (\ref{eq:GHE}) are $1-c$, $c-a-b$, $a-b$,
since Riemann's P-symbol is
\begin{equation}
\label{eq:PsymbolGHE}
P\left \{\begin{array}{ccc|c}
0 & 1 & \infty & z\\ \hline
0 & 0 & a & \\
1-c & c-a-b & b & 
\end{array} \right\}.
\end{equation}
Reflecting the generic set of 24 Kummer's $\hpgo21$ solutions of (\ref{eq:GHE}),
the notation $E(\alpha,\beta,\gamma)$ denotes any of the 24 corresponding hypergeometric equations.
Similarly, let $E(\alpha,\beta)$ denote a Fuchsian equation with two singularities and the exponent
differences $\alpha,\beta$. It exists if and only if $\alpha=\pm\beta$, 
formally coinciding then with $E(1,\alpha,\alpha)$. Otherwise, the equation $E(1,\alpha,\beta)$ has a logarithmic singularity.
Finally, let $\heunde{\alpha,\beta,\gamma,\delta}$ denote a Heun equation with its exponent differences
equal to $\alpha,\,\beta,\,\gamma,\,\delta$ in some order. 
The exponent differences for~(\ref{eq:HE}) are 
\begin{equation}
1-c, \quad 1-d, \quad c+d-a-b, \quad a-b,
\end{equation}
as evident from (\ref{eq:PsymbolHE}). The analogue of $24=2^{3-1}\cdot3!$ Kummer's 
hypergeometric solutions is the set of $192=2^{4-1}\cdot4!$ Maier's Heun solutions \cite{Maier10}. 
They are related by fractional-linear transformations that permute the 4 singular points and
interchange the local Heun solutions at them \cite[Appendix B]{Vidunas2009b}. 
The notation $\heunde{\alpha,\beta,\gamma,\delta}$ does not specify the parameters $t,q$,
hence it does not determine the monodromy.

Generally, a pull-back transformation of a hypergeometric equation $E(\alpha,\beta,\gamma)$ 
with respect to a covering  \mbox{$\varphi:\PP^1\to\PP^1$} has singularities 
at the branching points of $\varphi$ and 
above the singularities $z=0$, $z=1$, $z=\infty$ of  the hypergeometric equation. 
To have just 4 singularities after the pull-back (when $\deg \varphi>2$), we must have regular points
with $\varphi(x)\in\{0,1,\infty\}$.
As explained in \cite[\S 2]{HeunClass}, an $x$-point with $\varphi(x)\in\{0,1,\infty\}$ is regular
for the pulled-back equation only if the exponent difference of $E(\alpha,\beta,\gamma)$ at $\varphi(x)$
equals $\pm1/k$, where $k$ is the branching order of $\varphi$ at the $x$-point. 
The method \cite{HeunClass} to get hypergeometric-to-Heun transformations is to restrict 
some of $\alpha,\beta,\gamma$ to reciprocals of integers $k_z$, 
and look for coverings $\PP^1\to\PP^1$ with sufficiently many points of branching order $k_z$ 
in respective fibers. To leave a free parameter, at most two among $\alpha,\beta,\gamma$ are restricted.
We refer to the fibers of the singularities with restricted exponent differences as
{\em restricted fibers}. A useful consequence of the Hurwitz formula is this.
\begin{lemma}  \label{lm:hurwitz}
If $\varphi(x)$ is  a Belyi function of degree $d$, 
the number of distinct points with $\varphi(x)\in\{0,1,\infty\}$ equals $d+2$.
Otherwise the number of distinct points is greater. 
If the number of distinct points with $\varphi(x)\in\{0,1,\infty\}$ equals $d+3$,
there is exactly one branching point with $\varphi(x)\not\in\{0,1,\infty\}$,
and its branching order equals $2$.
\end{lemma}
\begin{proof} See  \cite[Lemma 2.2]{HeunClass}. Though the definition of Belyi function 
used in \cite{HeunClass} allows branching above any set of 3 points, not necessarily $\{0,1,\infty\}$,
its proof is valid.
\end{proof}
The following Diophantine inequality is derived in \cite[\S 3.1]{HeunClass} for pull-back functions
$\varphi(x)$ in hypergeometric-to-Heun transformations:
\begin{equation}
\frac{2}{d}+\sum_{z\in S} \,\frac1{k_z} \ge 1.
\end{equation}
Here $d=\deg\varphi$, all $k_z>0$, 
and $S\subset\{0,1,\infty\}$ is the set of points with restricted exponent differences. 
The tables in \cite[\S 3]{HeunClass} skip the following two cases:
\begin{itemize}
\item $|S|=1$ and the sole $k_z=1$. 
The singularity of $E(1,\alpha,\beta)$ with the exponent difference $1$
must be non-logarithmic (in order to have regular points in its fiber after the pull-back),
hence \mbox{$\beta=\pm\alpha$}. The equation $E(1,\alpha,\alpha)$ actually has just two singularities,
up to transformation (\ref{eq:algtransf}) with \mbox{$\varphi(x)\in\{x,1/x\}$}.
The monodromy representation is completely reducible, generated by one element (thus {\em cyclic}). 
\item  $|S|=2$ and both $k_z=2$. The monodromy of $E(1/2,1/2,\alpha)$ is a dihedral group;
it is infinite when $\alpha\not\in\QQ$. 
\end{itemize}
The hypergeometric equations $E(1,\alpha,\alpha)$ and $E(1/2,1/2,\alpha)$ have bases of Liouvillian solutions. The focus of the remaining sections is on pull-back transformation from them to Heun equations.
As we will see, there are pull-back transformations of any degree $d>1$ in both cases.
The pull-back coverings $\varphi(x)$ are Belyi functions, except
a series of composite coverings for pull-backs  from $E(1/2,1/2,\alpha)$,
described in Remark \ref{rm:nonb} here.
The obtained pull-back transformations give interesting parametric cases of Heun equations with 
Liouvillian solutions. 

\begin{remark} \rm
The full list of hypergeometric-to-Heun pull-backs with a continuous parameter and Liouvillian solutions
is formed by the transformations described in Theorems \ref{th:degen}, \ref{th:dihedral} here,
and Liouvillian specializations of the pull-backs P1, P2, P3, P15, P19, P20, P51 in \cite{Vidunas2009b} 
with more than one parameter.

Of particular interest are transformations between hypergeometric and Heun polynomials. 
For example, one may assume $a$ is a non-positive integer in (\ref{H-2F1 1tr}).
Here are instances of polynomial formulas induced by P15, P19 and P20:
\begin{eqnarray} \label{eq:c1}
\heun{1/4}{-9na/4}{-3n,\,3a}{\frac12;\,a-n+1/2}{x}\equal
\hpg{2}{1}{-n,\,a\,}{1/2}{\,x(4x-3)^2},\\
\heun{9}{\widehat{q}_1\!}{-3n,\,a-2n}{a-n+1/3;1-2n-2a}{x} \!\!\equal\!\!
(1-x)^{2n}\,\hpg{2}{1}{-n,\,a\,}{a-n+1/3}{\,-\frac{x(x-9)^2}{27(x-1)^2}},\\
\heun{9/8}{\widehat{q}_2\!}{-4n,\,a-3n}{3a-3n-1/2;a-n+1/2}{x} \!\!\equal\!\!\!
\left(1-\frac{8x}{9}\right)^{\!3n}\!\hpg{2}{1}{-n,\,a\,}{\!a-n+1/2}{\frac{64x^3(x-1)}{(8x-9)^3}}\!,\qquad
\end{eqnarray}
where $\widehat{q}_1=9na+18n^2-6n$, $\widehat{q}_2=-9na+9n^2+3n/2$. 
Reducible monodromy always leads to a polynomial (up to a power factor) solution.
A necessary reducibility condition for Heun's equation (\ref{eq:HE}) is 
$a,b,c-a,c-b,d-a,d-b,c+d-a$ or $c+d-b\in\ZZ$; check this in \cite{Maier10}. 

Non-reducible hypergeometric or Heun equations with a continuous parameter 
have dihedral monodromy. The hypergeometric equation is then $E(k+1/2,\ell+1/2,\alpha)$
with $k,\ell\in\ZZ$. The continuous parameter is preserved under these listed under
these transformations from \cite{Vidunas2009b}:
\begin{itemize}
\item the quadratic pull-back P1 transforms to $\heunde{k+\frac12,k+\frac12,2\ell+1,2\alpha}$
or $\heunde{2k+1,2\ell+1,\alpha,\alpha}$;
\item the quartic pull-back P2 transforms $E(\frac12,k+\frac12,\alpha)$ 
to $\heunde{2k+1,2k+1,2\alpha,2\alpha}$;
\item the quartic pull-back P3 transforms $E(\frac12,k+\frac12,\alpha)$ to 
 $\heunde{k+\frac12,k+\frac12,2k+1,4\alpha}$ or \mbox{$\heunde{4k+2,\alpha,\alpha,2\alpha}$};
 \item the cubic pull-back P15 transforms $E(\frac12,k+\frac12,\alpha)$ to 
 $\heunde{\frac12,k+\frac12,2k+1,3\alpha}$ or \mbox{$\heunde{\frac12,3k+\frac32,\alpha,2\alpha}$};
 \item the quartic pull-back P19 transforms $E(\frac12,k+\frac12,\alpha)$ 
to $\heunde{k+\frac12,3k+\frac32,\alpha,3\alpha}$.
\end{itemize}
This gives examples of Heun equations with cyclic or dihedral monodromy.
Additional such examples arise in pull-backs from general hypergeometric equations 
with cyclic monodromy \cite{Vidunas2007}, that is, $E(k,\alpha,\alpha+\ell)$  with $k,\ell\in\ZZ$, $|k|>|\ell|$. 
The continuous parameter is preserved only under the quadratic transformation P1, giving
 $\heunde{2k,\alpha,\alpha,2\alpha+2\ell}$ or $\heunde{k,k,2\alpha,2\alpha+2\ell}$.
 \end{remark}

\section{Transformations from the cyclic $\hpgde{1,\la,\la}$}
\label{sec:degenerate}

As already mentioned, the equations $\hpgde{1,\la,\la}$ are degenerate hypergeometric equations
with (essentially) two singularities. The monodromy representation is completely reducible, cyclic.
The 24 Kummer's solutions are constant or power functions, 
except 
\begin{equation} \label{reghpg}
\hpg{2}{1}{1-a,\,1}{2}{\,z\,}=\left\{ \begin{array}{cl} \displaystyle
\frac{1-(1-z)^{a}}{a\,z}, & \mbox{if } a\neq 0, \vspace{3pt}\\
\displaystyle -\frac1z\,\log(1-z), & \mbox{if } a=0. \end{array} \right.
\end{equation}
The similar formula in \cite[(29)]{VidunasFE} 
has to be corrected by the factor $-1$ in the $a\neq 0$ case.
As presented in \cite[\S 5]{VidunasFE}, there are pull-back transformations
$\hpgde{1,\la,\la}\pback{n}\hpgde{1,n\la,n\la}$. Up to M\"obius transformations,
they are the cyclic coverings $z\mapsto x^n$.  A nontrivial transformation formula is 
\begin{equation} \label{reghpgtr}
\hpg{2}{1}{1-na,\,1}{2}{\,x}=\psi(x)\,\hpg{2}{1}{1-a,\,1}{2}{nx\psi(x)}, 
\quad\mbox{with } \psi(x)=\frac{1-(1-x)^{n}}{n\,x}.
\end{equation}
The expression $\psi(x)$ is a polynomial of degree $n-1$.
The pull-back coverings for transforming $E(1,\alpha,\alpha)$ to Heun's equation
have similar expressions. Most of the Heun's solutions are trivial power functions.
\begin{theorem} \label{th:degen}
\begin{enumerate}
\item Pull-back transformations from $\hpgde{1,\la,\la}$ to Heun equations exist  and are unique
up to M\"obius transformations for any pair $(M,N)$ of positive integers. The transformed Heun equation 
is $\heunde{2, N\la, M\la, D\la}$, where $D=M+N$. 
\item The pull-back covering is a Belyi map of degree $D$. 
Up to M\"obius transformations, it is 
\begin{equation} \label{eq:degenphi}
\varphi(x) = 1-(1-x)^N\left(1+\frac{Nx}{M}\right)^M.
\end{equation}
\item The following identities with Heun's functions hold, for non-zero $a,b,a+b$:
\begin{eqnarray} \label{eq:cycheun1}
\heun{1+b/a}{-b\,(a+1)}{a,\,-b}{1+a;-1}{\,1-x\,}
 \!\equal \left(\frac{a\,x+b}{a+b}\right)^{\!b}, \\  \label{eq:cycheun2}
\heun{1+a/b}{-a\,(b+1)}{-a,\,b}{1+b;-1}{1+\frac{a x}{b}}
 \!\equal \left(\frac{a\,(1-x)}{a+b}\right)^{\!a}, \\
\heun{-a/b}{(b^2-a^2)\!\left(\frac{a-1}{b}+1\right)}{-a-b,\,2-a-b}{1-a-b;1-a}{\frac{1}{x}}
 \!\equal \left(1-\frac1x\right)^{\!a} \! \left(1+\frac{b}{a x}\right)^{\!b},\\ \label{eq:cycheun4}
\frac{a\,(a+b)}{2b}x^2\,\heun{-b/a}{2(1-b/a)}{2,\,2-a-b}{3;\,1-a}{x}  \!\equal
1-(1-x)^{a}\left(1+\frac{a x}{b}\right)^{\!b},
\end{eqnarray}
in neighborhoods of, respectively, $x=1$, $x=-b/a$, $x=\infty$, $x=0$.
These functions are solutions of $\heunde{2,a,b,a+b}$. This Heun equation
is a pull-back transformation of $\hpgde{1,\la,\la}$ if and only if the ratio $a/b$ is a rational number. 
The minimal degree of the pull-back  covering is equal to the sum of the numerator
and the denominator of $a/b$.
\end{enumerate}
\end{theorem}
\begin{proof} Let $D=\deg\varphi$.  Let $k$ denote the number of distinct points
above the two $\alpha$-points of $E(1,\alpha,\alpha)$. We have $k\in\{2,3,4\}$.
If $k=2$, the covering is cyclic and the pulled-back equation is $E(1,D\alpha,D\alpha)$.
If $k=4$, there are other branching points by Lemma {lm:hurwitz},
and those points would be excess singularities. 
If $k=3$, we are led to the branching pattern $M+N=D=2+[1]_{D-2}$ in the notation of \cite{HeunClass}.
By M\"obius transformations, we assign the points $x=\infty$, $x=0$, $x=1$ to the branching orders
$D$, $2$, $N$ (respectively) above $\varphi=\infty$, $\varphi=0$, $\varphi=1$ (respectively). 
The Belyi function has then the form $1-\varphi(x)=(1-x)^N(1-sx)^M$. The branching at $x=0$
gives $s=-N/M$.  Parts \refpart{i} and \refpart{ii} are shown.

Let $a=N\la$, $b=M\la$.
The formulas in \refpart{iii} are derived by writing down Heun-to-hypergeometric identities
similar to (\ref{H-2F1 1tr}) and (\ref{reghpgtr}), then evaluating the hypergeometric function using 
(\ref{reghpg}) and dropping the condition $a/b\in\QQ$ by the analytic continuation argument.
For example, the least trivial identity is
\begin{equation} \label{eq:trivpbf}
\heun{-M/N}{2(1-M/N)}{2,\,2-D\la}{3;\,1-N\la}{x}=
\frac{2M\,\varphi(x)}{ND\,x^2}\, 
\hpg21{1-\la,1}{2}{\varphi(x)}. 
\end{equation}
The formulas in \refpart{iii}  can be checked as follows. 
The Heun equation (\ref{eq:HE}) with  $(a,b,c,d,t,q)=$ \\
{$(0,-a-b,-1,1-a,-m/n,0)$} has this general solution:
\[
y(x)=C_1+C_2\big(x-1\big)^{a}\big(ax+b\big)^{b}.
\]
Formulas (\ref{eq:cycheun1})--(\ref{eq:cycheun4}) identify the most interesting Heun solutions
in the orbit of Maier's 192 Heun solutions \cite{Maier10}. 
The degree of a pull-back covering to $\heunde{2,a,b,a+b}$ with fixed $a/b\in\QQ$ 
must be a multiple of the sum of the numerator and the denominator.
\end{proof}

If the numbers $N,M$ are not co-prime, the covering of Theorem \ref{th:degen}
is the composition 
\begin{equation}
\hpgde{1,\la,\la}\pback{n}\hpgde{1,n\la,n\la}\pback{D/n}\heunde{2,N\la,M\la,D\la},
\end{equation}
where $n$ is a divisor of $\gcd(N,M)$. Putting $\alpha=1/D$ changes the pulled-back Heun equation
to the hypergeometric $\hpgde{2,N/D,M/D}$. Therefore the suitable Belyi covering
can be found in \cite[\S 5]{VidunasFE} as formula \cite[(33)]{VidunasFE} with 
$(k,\ell,n,m)=(D,N,0,1)$: 
\begin{eqnarray} \label{eq:trivpb}
  (1-z)  \mapsto  \frac{(1-x)^N}{\left(1-N x/D \right)^\degr}.
\end{eqnarray}
The covering $\varphi(x)$ in (\ref{eq:degenphi}) differs by the M\"obius transformation
$x\mapsto Dx/(Nx+M)$.
The pull-back covering 
can be written as
\begin{eqnarray} \label{phi:alt}
\varphi(x) = \frac{ND}{2M}\,x^2\,\heun{-M/N}{2(1-M/N)}{2,2-D}{3;1-N}{x}.
\end{eqnarray}
This is explained as follows. The specialization $\alpha=1$ gives a pull-back of the
the trivial "hypergeometric" equation $y''=0$ to $\heunde{2,N,M,D}$. We have here 
the Kleinian pull-back \cite{Klein1877} for $\heunde{2,N,M,D}$; its expression in terms
of Schwarz maps (as in \cite[\S 5]{VidunasHoeij}) gives (\ref{phi:alt}).

Identity (\ref{eq:trivpbf}) is valid for $\la=0$ as well, but then we must identify the Heun function as
\begin{equation}
-\frac{2M}{NDx^2}
\left(N\ln(1-x)+M\ln\left(1+\frac{Nx}M\right)\right).
\end{equation}

\begin{figure} \setlength{\unitlength}{1.3pt}
\[ \begin{picture}(230,116)(-16,-3)
\put(-10,90){\circle3} \put(8,90){\circle*3} \put(28,90){\circle3}  \put(48,90){\circle*3}
\put(-9,90){\line(1,0){36}}  \put(29,90){\line(1,0){19}} 
\put(8,108){\circle3}  \put(8,72){\circle3}  \put(8,73){\line(0,1){34}} 
\put(-5,77){\circle3} \put(-5,103){\circle3} \put(8,90){\line(-1,1){12}}  \put(8,90){\line(-1,-1){12}} 
\put(64,97){\circle3}  \put(64,83){\circle3}  \put(55,74){\circle3}  \put(55,106){\circle3}  
\put(48,90){\line(5,2){15}}  \put(48,90){\line(5,-2){15}}  
\put(48,90){\line(2,5){6}}  \put(48,90){\line(2,-5){6}}  
\put(112,90){\circle*3}  \put(130,108){\circle*3} \put(130,72){\circle*3}  \put(148,90){\circle*3}
\put(117,103){\circle3}  \put(117,77){\circle3} \put(143,103){\circle3}  \put(143,77){\circle3}
\qbezier(116,78)(108,90)(116,102) \qbezier(144,78)(152,90)(144,102)
\qbezier(118,76)(130,68)(142,76)  \qbezier(118,104)(130,112)(142,104)
\put(157,106){\circle3}  \put(157,74){\circle3} \put(184,90){\circle3}  
\put(175,106){\circle*3}  \put(175,74){\circle*3}
\qbezier(156,75)(140,90)(156,105)  \qbezier(158,73)(179,67)(184,89) 
\qbezier(158,107)(179,113)(184,91) 
\put(14,30){\circle*3}  \put(32,48){\circle*3} \put(32,12){\circle*3}  \put(50,30){\circle*3}
\put(19,43){\circle3}  \put(19,17){\circle3} \put(45,43){\circle3}  \put(45,17){\circle3}
\qbezier(18,18)(10,30)(18,42) \qbezier(46,18)(54,30)(46,42)
\qbezier(20,16)(32,8)(44,16)  \qbezier(20,44)(32,52)(44,44)
\put(68,30){\circle3}  \put(86,30){\circle*3} 
\put(50,30){\line(1,0){17}} \put(69,30){\line(1,0){17}}
\put(132,30){\circle*3}  \put(150,48){\circle*3} \put(150,12){\circle*3}  \put(168,30){\circle*3}
\put(137,43){\circle3}  \put(137,17){\circle3} \put(163,43){\circle3}  \put(163,17){\circle3}
\qbezier(136,18)(128,30)(136,42) \qbezier(164,18)(172,30)(164,42)
\qbezier(138,16)(150,8)(162,16)  \qbezier(138,44)(150,52)(162,44)
\put(186,30){\circle3}  \put(204,30){\circle*3}  \put(222,30){\circle3} 
\put(168,30){\line(1,0){17}} \put(187,30){\line(1,0){34}}
\put(-19,59){\refpart{a}}  \put(99,59){\refpart{b}} 
\put(1,2){\refpart{c}} \put(119,2){\refpart{d}} 
\end{picture}  \]
\caption{Dessins d'enfant for Belyi pull-backs to $\heunde{2, N\la, M\la, D\la}$ 
and $\heunde{1/2, 3/2, N\la, M\la}$}
\label{fig:dessins}
\end{figure}

\begin{remark} \rm
As is known \cite{Schneps94}, Belyi functions (up to M\"obius transformations) are in a bijective
correspondence with {\em dessins d'enfant} (up to homotopy). The latter are certain bi-coloured graphs
on the Riemann sphere (in our case) such that the incidence degrees of cells and vertices 
of both colors follow the branching pattern. Up to homotopy, a dessin d'enfant of $\varphi(x)$
is the pre-image of the real interval $[0,1]$ with the points above $z=0$, $z=1$ coloured (say)
black and white, respectively. The points above $z=\infty$ are represented by the cells. 
The dessin is a tree when there is a unique cell (of degree $D$).
In particular, the dessin for (\ref{eq:degenphi}) is depicted in Figure \ref{fig:dessins}\refpart{a}.
The number of ``sticks" coming out of the black vertices is $N-1$ and $M-1$.
It is straightforward to see that no other dessins are possible for its branching pattern,
confirming the uniqueness of $\varphi(x)$.
\end{remark}

\begin{remark} \label{rm:cyclic} \rm
Formulas (\ref{eq:cycheun1})--(\ref{eq:cycheun2}) are fancy expressions for the power 
function $X^b$, yet with a free parameter $a$. A M\"obius transformation gives
the straightforward expression
\begin{equation}
\heun{1+b/a}{-b\,(a+1)}{a,\,-b}{1+a;-1}{\frac{a+b}{a}\,x} =  \left(1-x\right)^{b}.
\end{equation}
A fractional-linear transformation from \cite[(B.3)]{Vidunas2009b} 
and the change $b\mapsto-b$ give the symmetric expression
\begin{equation}
\heun{a/(a-b)}{\frac{a\,b\,(a+1)}{a-b}}{a,\,b}{1+a;1+b}{\,x} =  \left(1-x\right)^{-b}.
\end{equation}
Since one of the singularities of the considered Heun equation is apparent with the local exponent
difference 2, expressions in terms of $\hpgo32$ functions in 
\cite[\S 5]{MaierPH} apply. The result (with $e=a$ in \cite[Theorem 5.3]{MaierPH}) is trivial:
\[ 
\left(1-x\right)^{-b}=\hpg32{a,b,a+1}{a+1,a}{x}.
\] 
However, the expression in terms of contiguous $\hpgo21$ functions in 
\cite[Corollary 5.3.1]{MaierPH} is
\begin{equation}
\left(1-x\right)^{-b}=\hpg21{a,b}{a+1}{x}+\frac{b\,x}{a+1}\;\hpg21{a+1,b+1}{a+2}{x}.
\end{equation}
This relates $\heunde{2,\la,\lb,\la+\lb}$ to the contiguous orbit of $\hpgde{\la,\lb+1,\la+\lb}$.
Interestingly, the monodromy representation for the Heun equation is completely reducible,
while the monodromy representation for the contiguous hypergeometric functions has (generally)
only one invariant subspace of dimension 1.
\end{remark}

\section{Transformations from dihedral $\hpgde{1/2,1/2,\la}$}
\label{sec:dihedral}

The monodromy representation of 
$\hpgde{1/2,1/2,\la}$ with general $\alpha\in\CC$ is an infinite dihedral group, 
hence these hypergeometric equations and their solutions are 
said to be {\em dihedral}. 
Their hypergeometric solutions are very explicit:
\begin{eqnarray}   \label{dihedr1}
\hpg{2}{1}{\frac{a}{2},\,\frac{a+1}{2}\,}{a+1}{\,z} \equal
\left(\frac{1+\sqrt{1-z}}{2} \right)^{-a},\\  \label{dihedr2}
\hpg{2}{1}{\frac{a}{2},\,\frac{a+1}{2}\,}{\frac{1}{2}}{\,z} \equal
\frac{(1-\sqrt{z})^{-a}+(1+\sqrt{z})^{-a}}{2}. 
\end{eqnarray}
General expressions for contiguous $\hpgo21$ functions are presented in \cite{Vidunas2008a}.
As shown in  \cite[\S 4]{VidunasDihTr}, there are pull-back transformations 
$\hpgde{1/2,1/2,\la}\pback{d}\hpgde{1/2,1/2,n\la}$ of any degree $d$.
The pull-back coverings have the form $x\mapsto x\theta_2(x)^2/\theta_1(x)^2$,
where 
\begin{eqnarray}
\theta_1(x)=\hpg21{-\frac{n}2,-\frac{n-1}2}{1/2}{x},\qquad
\theta_2(x)=n\,\hpg21{-\frac{n-1}2,-\frac{n-2}2}{3/2}{x}.
\end{eqnarray}
The components $\theta_1(x)$, $\theta_2(x)$ can be expressed in terms of Chebyshev polynomials 
$T_n(z)$, $U_{n-1}(z)$, respectively.  
The identity $(1-\sqrt{x})^n=\theta_1(x)-\theta_2(x)\sqrt{x}$ holds.
These transformations can be composed with the simplification 
$\hpgde{1/2,1/2,\la}\pback{2}\hpgde{1,\la,\la}$ to a cyclic monodromy group.

\begin{remark} \label{rm:nonb} \rm
Pull-back transformations of $\hpgde{1/2,1/2,\la}$ 
to Heun equations may involve non-Belyi 
coverings. 
An example is the composition
\begin{equation} \label{eq:dihcomp}
\hpgde{1/2,1/2,\la}\pback2 
\hpgde{1,\la,\la}\stackrel{N+M}{\xleftarrow{\hspace*{24pt}}}\heunde{2,N\la,M\la,(N+M)\la},
\end{equation}
Here the second transformation is as in Theorem $\ref{th:degen}$, 
but the fiber of the ``singularity" with the exponent difference 1 may be an arbitrary point 
of $\PP^1_z\setminus\{0,1,\infty\}$.
For instance, taking $N=M=1$ we derive the covering 
\begin{equation}
\varphi_s(x)=\frac{4s\,x\,(2-x)}{(x^2-2x-s)^2}
\end{equation}
with a parameter $s$. The branching points lie above 4 points, 
namely $\varphi(x)\in\{0,1,\infty,4s/(s+1)^2\}$. 
The transformed Fuchsian equation for
\begin{equation}
\left(1+\frac{2x-x^2}{s}\right)^{\!-e}\,
\hpg21{\frac{e}2,\frac{e+1}2}{1+e}{\frac{4s\,x\,(2-x)}{(x^2-2x-s)^2}}
\end{equation}
is Heun's equation (\ref{eq:HE}) with $(t,q,a,b,c,d)=(2,0,0,2e,1+e,-1)$.
The specialized coverings $\varphi_s(x)$ with $s=\pm1$ are Belyi coverings. 
According to \cite[Proposition 3.3]{HeunClass}, 
the only other case when non-Belyi maps occur in Heun-to-hypergeometric  
transformations is when the Fuchsian equations have a basis of algebraic solutions.
\end{remark} 
\begin{theorem} \label{th:dihedral}
\begin{enumerate}
\item The pull-back transformations from $\hpgde{1/2,1/2,\la}$ to Heun equations 
with respect to a non-Belyi covering $\varphi(x)$ are compositions $(\ref{eq:dihcomp})$.
The non-Belyi coverings are parametric, and may specialize to Belyi coverings.
\item Pull-back transformations from $\hpgde{1/2,1/2,\la}$ to Heun equations 
with respect to other Belyi coverings $\varphi(x)$ exist  
and are unique up to M\"obius transformations 
for any pair $(M,N)$ of non-equal positive integers, $M\neq N$. 
The transformed Heun equation is $\heunde{1/2,3/2,N\la,M\la}$.
\item The degree of the mentioned Belyi covering $\varphi(x)$ is $D=M+N$. 
Up to M\"obius transformations, the Belyi covering is $\varphi(x)=x^3\,\Theta_2(x)^2/\Theta_1(x)^2$,
where $\Theta_1(x),\Theta_2(x)$ are the polynomials in $x$ determined by
\begin{equation} \label{eq:dihex}
\left(1+\sqrt{x}\right)^N\left(1-\frac{N\sqrt{x}}M\right)^{\!M}=\Theta_1(x)+x^{3/2}\,\Theta_2(x).
\end{equation}
\item The polynomials $\Theta_1(x)$, $\Theta_2(x)$ are Heun polynomials:
\begin{eqnarray}
\Theta_1(x)\equal \heun{M^2/N^2}{MD/4N} 
{-\frac{D}2,-\frac{D-1}2}{-\frac12;\,1-N}{x},\\
\Theta_2(x)\equal \frac{ND(M-N)}{3M^2} \,
\heun{M^2/N^2}{\frac32\left(1+\frac{M^2}{N^2}\right)-\frac{5MD}{4N}}
{-\frac{D-3}2,-\frac{D-4}2}{\frac52;\,1-N}{x}.
\end{eqnarray}
\end{enumerate}
\end{theorem}
\begin{proof} 
The non-singular points above the singularities of  $E(1/2,1/2,\alpha)$ are the simple branching
points in the two fibers with the exponent difference $1/2$. Let $k$ denote the number of simple 
branching points in those two fibers, and let $\ell$ denote the number of points in the third fiber. 
We have $k\le d$, \mbox{$\ell\le 4$}.  Besides, $k\ge d-2$ by Lemma \ref{lm:hurwitz}.
The case $(k,\ell)=(d-1,2)$ leads to too many singularities.
There are too few points ($<d+2)$ in the three fibers when \mbox{$(k,\ell)\in\{(d,1),(d-2,1)\}$}. 
The cases \mbox{$(k,\ell)\in\{(d,2),(d-1,1)\}$} lead 
to hypergeometric transformations to $E(1,\beta,\beta)$ and $E(1/2,1/2,\beta)$.
The case $(k,\ell)=(d-2,2)$ leads to the branching patterns
$[2]_{d/2-1}=[2]_{(d-4)/2}+3+1=N+M$ and $[2]_{(d-1)/2}+1=[2]_{(d-3)/2}+3=N+M$.
We can have $\ell=4$ only with $k=d$, but then we have more
singularities outside the 3 fibers. With $\ell=3$, we can obtain non-Belyi coverings 
with the branching $[2]_{d/2}=[2]_{d/2}=K+N+M$ in the three fibers,
and Belyi coverings with branching patterns $[2]_{d/2}=[2]_{(d-4)/2}+4=K+N+M$.

Transformations to Heun equations are obtained when $(k,\ell)=(d-2,2)$ or $\ell=3$.

Consider first the non-Belyi coverings in the $\ell=3$ case. Alter the specialization $\alpha=1/K$,
they would pull-back $E(1/2,1/2,\alpha)$ to $E(2,N/K,M/K)$. 
By \cite[\S 7]{Vidunas2007}, the latter equation has a completely 
reducibly monodromy (rather than logarithmic singularities) only when $N/K\pm M/K\in\{0,\pm1\}$.
The case $N=M$ is ruled out by the specialization $\alpha=1/N$, as it leads to
a pull-back to $E(2,K/N)$ that exists only when $K=2N$; but that is an instance of
the $N/K+M/K=1$ case. We conclude that one of the numbers $K,N,M$ is the sum 
of the other two. The specialization $\alpha=1$ gives a pull-back to $\heunde{2,N,M,K}$
with the monodromy smaller than $\ZZ/2\ZZ$. The functions fields must factor through
the maximal $\ZZ/2\ZZ$-invariant field, hence the factorization 
$\hpgde{1/2,1/2,1}\pback2 \hpgde{1,1,1}\pback{d/2}\heunde{2,N,M,K}$, giving (\ref{eq:dihcomp}).
As already mentioned in Remark \ref{rm:nonb}, the non-Belyi coverings have a parameter.

The same specialization $\alpha\in\{1/K,1/N,1\}$ arguments apply to the Belyi coverings with $\ell=3$. 
Their dessins d'enfant look like in Figure \ref{fig:dessins}\refpart{b}.
They are a special case of the one-parameter non-Belyi covering.

Finally, the Belyi coverings with $(k,\ell)=(d-2,2)$ give the transformations
$\hpgde{1/2,1/2,\la}\pback{d}\heunde{1/2,3/2,N\la,M\la}$, with $d=N+M$.
The specialization $\la=1/M$ gives a transformation to $\heunde{1/2,3/2,N/M}$,
so we can apply 
\cite[Theorem 5.1]{VidunasDihTr} with $(k,\ell,m,n)=(1,0,M,N)$.
The obtained pullback covering is described in \refpart{iii}. 
In particular, multiplying (\ref{eq:dihex}) with its $\sqrt{x}$-conjugate gives
\begin{equation}
\left(1-x\right)^N\left(1-\frac{N^2\,x}{M^2}\right)^{\!M}=\Theta_1(x)^2-x^{3}\,\Theta_2(x)^2,
\end{equation}
convincing us that $\varphi(x)$ is a Belyi covering. 
We see that the fourth singular point of the pull-backed Heun equation is $t=M^2/N^2$.
Besides, $\Psi(x)$ of \cite[\S 5]{VidunasDihTr} is equal, up to a constant multiple, 
to \mbox{$N^2x-M^2$}. Theorem 5.6 in \cite{VidunasDihTr} gives differential equations for
$\Theta_1(x),\Theta_2(x)$, and there we recognize Heun's equations $\heunde{1/2,3/2,N,M}$
in our particular setting. This allows us to identify the expressions in \refpart{iv}.
The covering $\varphi(x)$ degenerates when $M=N$. There is no pullback transformation
$\hpgde{1/2,1/2,\la}\pback{2N}\heunde{1/2,3/2,N\la,N\la}$ at all,
because the specialization $\la=1/N$ gives a pull-back to the non-existent  $E(1/2,3/2)$.
\end{proof}

As noticed in \cite[\S 5.2]{VidunasDihTr}, for $N=1$ we have
\begin{eqnarray*}
\Theta_1(x)=\hpg21{-\frac{M}2,-\frac{M+1}2}{-1/2}{\frac{x}{M^2}},\qquad
\Theta_2(x)=\frac{M^2-1}{3M^2}\,\hpg21{-\frac{M-2}2,-\frac{M-3}2}{5/2}{\frac{x}{M^2}},
\end{eqnarray*}
and for $N=2$ we can write $\Theta_1(x),\Theta_2(x)$ as $\hpgo32$ polynomials.
The dessin d'enfant for the Belyi covering of \refpart{iii} is as depicted
in Figure \ref{fig:dessins}\refpart{c} or \refpart{d}. The degree of the bounded cell is $\min(M,N)$.

Like in Theorem (\ref{th:degen})\refpart{iii}, Heun solutions of the pulled-back dihedral equation
can be expressed by writing down Heun-to-hypergeometric identities
and evaluating the hypergeometric function using formulas like (\ref{dihedr1})--(\ref{dihedr2}).
After setting $a=N\alpha$, $b=M\alpha$ and dropping the condition $a/b\in\QQ$ 
by the analytic continuation argument, we get the following formulas:
\begin{eqnarray}
\heun{b^2/a^2}{-\frac{b(a+b)}{4a}}{\frac{a+b}2,\,\frac{a+b+1}2}{-\frac12;1+a}{x}\equal 
\frac{\big(1+\sqrt{x}\big)^{-a}\left(1-\frac{a}{b}\sqrt{x}\right)^{-b}
+\big(1-\sqrt{x}\big)^{-a}\left(1+\frac{a}{b}\sqrt{x}\right)^{b}}2, \qquad\\
\heun{\frac{a^2-b^2}{a^2}}{\frac{(a+b)^2(a+1)}{4a}}{\frac{a+b}2,\frac{a+b+1}2}{1+a;-\frac12}{1-x}
\equal \left(\frac{1+\sqrt{x}}2\right)^{-a}\left(\frac{b-a\sqrt{x}}{b-a}\right)^{-b}.
\end{eqnarray}
Another formula is
\begin{eqnarray}
&& \heun{b^2/a^2}{\frac{5ab(a+b)+6(a^2+b^2)}{4a^2}}
{\frac{a+b+3}2,\,\frac{a+b+4}2}{\frac52;\,1+a}{x}= \nonumber\\
&& \hspace{60pt} \frac{3b^2\,x^{-3/2}}{2a(a^2-b^2)} \textstyle
\left( \big(1+\sqrt{x}\big)^{-a}\left(1-\frac{a}{b}\sqrt{x}\right)^{-b}
-\big(1-\sqrt{x}\big)^{-a}\left(1+\frac{a}{b}\sqrt{x}\right)^{-b}\right). \qquad
\end{eqnarray}
The Heun solutions of the same equation with the argument $1/x$ evaluate as follows,
after the change $x\mapsto 1/x$:
\begin{eqnarray}
& & \heun{a^2/b^2}{\frac{(a^2-b^2)^2+a^3+b^3}{4b^2}}{\frac{a+b}2,\,\frac{a+b+3}2}{\frac12;\,1+a}{x}
= \frac{\big(1+\sqrt{x}\big)^{-a}\left(1-\frac{b}{a}\sqrt{x}\right)^{-b}
+\big(1-\sqrt{x}\big)^{-a}\left(1+\frac{b}{a}\sqrt{x}\right)^{b}}2, \qquad \\
& & \heun{a^2/b^2}{\frac{(a^2-b^2)^2+3(a^3+b^3)+2(a^2+b^2)}{4b^2}}
{\frac{a+b+1}2,\,\frac{a+b+4}2}{\frac32;\,1+a}{x} = \nonumber \\
& & \hspace{112pt} \frac{a\,x^{-1/2}}{2(b^2-a^2)} \textstyle
\left( \big(1+\sqrt{x}\big)^{-a}\left(1-\frac{b}{a}\sqrt{x}\right)^{-b}
-\big(1-\sqrt{x}\big)^{-a}\left(1+\frac{b}{a}\sqrt{x}\right)^{-b}\right).
\end{eqnarray}
These formulas can be applied to evaluate all solutions of $\heunde{1/2,3/2,a,b}$.


\begin{thebibliography}{99}

\bibitem{VidunasHoeij} M. van Hoeij, R. Vid\=unas,  
\emph{Belyi coverings for hyperbolic Heun-to-hypergeometric transformations}.
Available at {\sf http://arxiv.org/abs/1212.3803v1}.

\bibitem{Klein1877} 
F. Klein, 
\newblock Uber lineare {D}ifferentialgleichungen {I}, 
\newblock Math. Annalen 11 (1877), 115--118.

\bibitem{Kovacic}
J. J. Kovacic,
\newblock An algorithm for solving second order linear differential equations,
\newblock Journ. Symb. Comp. 2 (1986), 3--43.

\bibitem{MaierPH} R. S. Maier,    $P${\emph-symbols, Heun identities, and ${}_3F_2$ identities},
in: \emph{Special functions and Orthogonal Polynomials},
eds.~D.~Dominici and R.~S.~Maier (Contemporary Mathematics series,
No 471), AMS, Providence, 2007,  139--160.

\bibitem{Maier10}
R.~S. Maier.
\newblock The 192 solutions of the {H}eun equation.
\newblock {\em Math. Comp.}, 76 (2007), 811--843.

\bibitem{Ronveaux95}
A.~Ronveaux, editor.
\newblock {\em Heun's Differential Equations}.
\newblock Oxford Univ. Press, Oxford, UK, 1995.
\newblock With contributions by F.~M. Arscott, S.~Yu. Slavyanov, D.~Schmidt,
  G.~Wolf, P.~Maroni and A.~Duval.

\bibitem{Schneps94}
L.~Schneps.
\newblock Dessins d'enfants on the {R}iemann sphere.
\newblock In L.~Schneps, editor, {\em The {G}rothendieck Theory of Dessins
  d'Enfants}, number 200 in London Mathematical Society Lecture Note Series,
  pages 47--77. Cambridge Univ. Press, Cambridge, UK, 1994.

\bibitem{Vidunas2007} 
 R. Vidunas,
 \newblock Degenerate Gauss hypergeometric functions,   
 \newblock Kyushu J.   Math.  61 (2007), 109--135.
 
 \bibitem{VidunasFE}
R.~Vidunas.
\newblock Algebraic transformations of {G}auss hypergeometric functions.
\newblock {\em Funkcial. Ekvac.}, 52 (2009), 139--180.

\bibitem{Vidunas2008a}
R.~Vidunas.
\newblock Dihedral {G}auss hypergeometric functions.
\newblock Kyushu J. Math. 65 (2011), 141--167.

\bibitem{VidunasDihTr} R. Vidunas, \emph{Transformations and invariants
for dihedral Gauss hypergeometric functions},
\newblock Kyushu J. Math. 66 (2012), 143--170.

\bibitem{HeunClass} R. Vidunas, G. Filipuk,
\emph{A classification of coverings yielding Heun-to-hypergeometric reductions}.
Accepted by {\em Osaka J.~Math.}. Available at {\sf http://arxiv.org/abs/1204.2730}.

\bibitem{Vidunas2009b}
R.~Vidunas, G.~Filipuk.
\newblock Parametric transformations between the {H}eun and {G}auss 
  hypergeometric functions.
\newblock Funkcialaj Ekvacioj, 56 (2013), 271--321.

\end{thebibliography}
\end{document}